\documentclass[reqno, dvipsnames, oneside]{amsart}
\usepackage[utf8]{inputenc}
\usepackage{amscd, amssymb, amsthm, mathrsfs, amsmath, amsrefs, amsfonts, graphicx}
\usepackage{graphicx}
\usepackage{xcolor}
\usepackage[margin=1.2in]{geometry}
\usepackage{hyperref}
\usepackage{amssymb}
\usepackage{todonotes}
\usepackage[all]{xy}
\xyoption{all}
\usepackage{multicol}
\def\H{\mathcal{H}}
\def\U{\mathcal{U}}

\def\C{\mathbb{C}}
\def\G{\mathcal{G}}

\newtheorem{prop}[equation]{Proposition}
\newtheorem{lema}[equation]{Lemma} 
\newtheorem{teo}[equation]{Theorem}
\newtheorem{rem}[equation]{Remark}
\newtheorem{coro}[equation]{Corollary}
\newtheorem{defi}[equation]{Definition}
\newtheorem{exa}[equation]{Example}

\numberwithin{equation}{section}
\theoremstyle{plain}
\theoremstyle{definition}

\theoremstyle{remark}

\title{An imprimitivity theorem for finite algebraic quantum groups}
\author[1]{Eugenia Ellis}
\author[2]{Ana Gonz\'alez}
\author[3]{Gisela Tartaglia}
\thanks{The first author was partially supported by ANII, CSIC and PEDECIBA.  The second author was partially supported by ANII and CSIC. The first and third authors were partially supported by grant PICT-2021-I-A-0 0710. The third author was partially supported by grant PIP 11220200100423CO}
\address[1,2]{IMERL, Fac. de Ingenier\'ia, Universidad de la República, Montevideo, Uruguay}
\address[3]{CMaLP - CONICET, FCE-UNLP, La Plata, Argentina}
\email[1]{eellis@fing.edu.uy}
\email[2]{anagon@fing.edu.uy}
\email[3]{gtartaglia@mate.unlp.edu.ar}
\date{}

\begin{document}
\begin{abstract}
    Let $\mathcal{G}$ be an algebraic quantum group and $\U$ a compact quantum subgroup. Given a left $\hat{\U}$-module algebra $A$ with unit, we can endow  $A\otimes\mathcal{G}$ with a structure of a right $\hat{\U}$-module algebra. The algebra of invariants for this action $(A\otimes\mathcal{G})^{\hat{\U}}$ has a left action of $\hat{\mathcal{G}}$. We prove that for finite $\G$, $(A\otimes\mathcal{G})^{\hat{\U}}\#\hat{\mathcal{G}}$ is Morita equivalent to $A\#\hat{\mathcal{U}}$.
\end{abstract}

\maketitle

\section{Introduction}
Let $\mathcal{G}$ be an {\emph{algebraic quantum group}} in the sense of \cite{VD} and \cite{DVDZ}, that is, a regular multiplier Hopf algebra with invariants. A {\emph{finite algebraic quantum group}} is a finite dimensional algebraic quantum group, in other words, a finite dimensional Hopf algebra. 
The definition that we consider here differs from the one used in \cite{VD3}, since we do not require that the underlying algebras be equipped with an involution.

Algebraic quantum groups have a good duality framework. Every algebraic quantum group $\mathcal{G}$ has a dual $\hat{\mathcal{G}}$, which is again an algebraic quantum group and satisfies $\hat{\hat{\mathcal{G}}}\cong \mathcal{G} $. This is a generalization of the duality for finite dimensional Hopf algebras. An algebraic quantum group $\mathcal{G}$ is \emph{compact} if it has a unit.  Dually, $\mathcal{G}$ is \emph{discrete} if $\hat{\mathcal{G}}$ has unit.  We say that another algebraic quantum group $\mathcal{U}$ is a \emph{compact subgroup} of $\mathcal{G}$ if it is compact, unimodular and there exists a surjective morphism $\pi: \mathcal{G} \rightarrow \mathcal{U}$.

In this work we present an imprimitivity theorem for finite algebraic quantum groups. 
Recall from \cite{VDZ1} that $A$ is a left $\G$-module algebra if it satisfies the following conditions:
\begin{itemize}
    \item $A$ is a unital left $\G$-module, here $A$ unital means $\G A=A$;
    \item $x\rightharpoonup aa'= \sum (x_1\rightharpoonup a)(x_2\rightharpoonup a')$, $\forall x\in \G, a, a'\in A$.
\end{itemize}

The main result of this work is

\begin{teo} (Theorem \ref{morita})
  Let $\H$ be finite algebraic quantum group, $\U$ a compact quantum subgroup and $A$ a left $\hat{\U}$-module algebra with unit. The algebras $A\#\hat{\U}$ and $(A\otimes\H)^{\hat{\U}}\#\hat{\H}$ are Morita equivalent, where  $(A\otimes\H)^{\hat{\U}}$ is the invariant subalgebra for the following action:
\[
a\otimes h\leftharpoonup\beta=\sum \overline{S}(\beta_1)\rightharpoonup a\otimes h\leftharpoonup \beta_2.
\]
 \end{teo}

This is a imprimitivity theorem for a particular case of compact algebraic quantum groups, the finite dimensional one. A discrete version of Green's imprimitivity theorem for a particular case of discrete algebraic quantum groups is presented in \cite{CE}*{Section 10.4}, that is for $\hat{\mathcal{G}}=\mathbb{C}G$ a group ring.

The paper is organized as follows. In Section \ref{sec1} we review from \cite{DVDZ}, \cite{T}, \cite{VD2}, \cite{VD}, \cite{VD1} and \cite{VDZ1} the definition of an algebraic quantum group, the construction of its dual and the notions of action and coaction. We also introduce the notion of compact quantum subgroup (Definition \ref{cqs}). Given an algebraic quantum group $\G$, a compact quantum subgroup $\U$ and a left $\hat{\U}$-module algebra $A$ with unit, we endow $A\otimes\G$ with the structure of right $\hat{\U}$-module algebra. Then we show how to define a left $\hat{\G}$-action on the algebra of invariants $(A\otimes\G)^{\hat{\U}}$. In Section \ref{sec2} we consider the discrete quantum group $C_c(G)$ given by finitely supported complex valued functions on a group $G$. In this case, if $H\subseteq G$ is a finite subgroup, $C_c(H)$ is a compact quantum subgroup of $C_c(G)$. Given a left $\widehat{C_c(H)}=\C H$-module algebra $A$ with unit, we use the actions defined in the previous section to endow $A\otimes C_c(G)=C_c(G,A)$ with the structure of right $\C H$-module algebra. The algebra of invariants $C(G,A)^{\C H}$ for this action, which is a left $\widehat{C_c(G)}=\C G$-module algebra, coincides with the $G$-algebra $\operatorname{BigInd}_H^G(A)$ obtained in \cite{CE}. Then, following the ideas of \cite{CE}, we restrict our attention to the $\C G$- subalgebra $C_c(G,A)^{\C H}\subseteq C(G,A)^{\C H}$. It turns out that for $H$ finite, $C_c(G,A)^{\C H}$ identifies with the $G$-algebra $\operatorname{Ind}_H^G(A) \subseteq \operatorname{BigInd}_H^G(A)$.
After reviewing the concept of Morita equivalence in the context of rings with local units, in Proposition \ref{moritacoeff} we construct a Morita context between 
$C_c(G,A)^{\C H}\#\C G$ and $A\# \C H$, and show that it is an equivalence (see also \cite{CE}*{Theorem 10.4.5}).
In Section \ref{sec3} we study the case of finite algebraic quantum groups, i.e. finite dimensional Hopf algebras. In this context, a compact quantum subgroup of $\H$ is given by a cosemisimple finite dimensional Hopf algebra $\U$, equipped with a surjective homomorphism $\pi:\H\to\U$. Given a left $\hat{\U}$-module algebra $A$ with unit, we apply the actions defined in Section \ref{sec1} to induce a left $\hat{\H}$-action on the algebra of invariants $(A\otimes\H)^{\hat{\U}}$. 
After some preliminary results, we prove Theorem \ref{morita}, which states that there is a Morita equivalence between $(A\otimes\H)^{\hat{\U}}\#\hat{\H}$ and $A\#\hat{\U}$.  

\smallskip
The authors wish to thank Stefan Vaes for his helpful comments.

\smallskip
\textbf{Notation.} All vector spaces and algebras are assumed to be over $\C$.  We write $\otimes$ for the algebraic tensor product $\otimes_{\C}$ and $\iota_X$ for the identity map of the space $X$. Given a vector space $V$, we put $V^*$ for the space of linear functionals $Hom_{\C}(V,\C)$. The left/right actions of a ring $R$ on an algebra $A$ will be denoted by $r\rightharpoonup a$ and $a\leftharpoonup r$, respectively. If $(H,\Delta)$ is a Hopf algebra, we will use Sweedler notation and write $\Delta(h)=\sum  h_1\otimes h_2$.
\section {Algebraic quantum groups} \label{sec1}

In this Section we review from \cite{DVDZ,T, VD2, VD, VD1, VDZ1} the definition of an algebraic quantum group, the construction of its dual and the notions of action, coaction and invariant subalgebra. We also establish the notion of algebraic compact quantum subgroup that we will use in the following sections.
\subsection{Multipier Hopf algebras with invariants}
Let $\G$ be an algebra with non-degenerate product (i.e. $\G x=0$ implies $x=0$ and $y\G=0$ implies $y=0$). The multiplier algebra $M(\G)$ is the largest unital algebra that contains $\G$ as an essential two-sided ideal. Given another algebra with non degenerate product $\H$, a homomorphism $\alpha:\G\to M(\H)$ is called \emph{non-degenerate} if the linear span of $\alpha(\G)\H$ and the linear span of $\H\alpha(\G)$ both are equal to $\H$. In this case, $\alpha$ has a unique extension to $M(\G)$ (\cite{VD2}*{Prop.A.5}). We will also write $\alpha$ for the extension.

\smallskip
An algebra $\G$ with non-degenerate product is called  \emph{multiplier bialgebra} if it is equipped with a non-degenerate homomorphism $\Delta:\G\to M(\G\otimes \G)$ such that:
\begin{enumerate}
    \item $\Delta(\G)(1\otimes \G)$ and $(\G\otimes 1)\Delta(\G)$ are contained in $\G\otimes \G$;\\ here $1$ denotes the unit of $M(\G)$, and the spaces $1\otimes \G$ and  $\G\otimes 1$ are considered as subsets of $M(\G)\otimes M(\G)\subseteq M(\G\otimes \G)$.
    \item $\Delta$ is coassociative: $(\iota_{\G}\otimes\Delta)\circ\Delta=(\Delta\otimes \iota_{\G})\circ\Delta$;\\ here the homomorphisms $\iota_{\G}\otimes\Delta$ and $\Delta\otimes\iota_{\G}$ have been extended to $M(\G\otimes \G)$.
    \end{enumerate}
    A morphism of multiplier bialgebras from $(\G,\Delta_{\G})$ to $(\H,\Delta_{\H})$ is a non-degenerate homomorphism $\alpha:\G\to M(\H)$ such that $\Delta_{\H}\circ \alpha=(\alpha\otimes\alpha)\circ\Delta_{\G}$.
\smallskip
A multiplier bialgebra $(\G,\Delta)$ is a \emph{multiplier Hopf algebra} if the linear maps $T_1,T_2:\G\otimes \G \to \G\otimes \G$ given by
\[T_1(x\otimes y)=\Delta(x)(1\otimes y) \qquad T_2(x\otimes y)=(x\otimes 1)\Delta(y) \]
are bijective. A morphism of multiplier Hopf algebras is a morphism of the underlying multiplier bialgebras. For any multiplier Hopf algebra $(\G,\Delta)$ there are unique linear maps $\epsilon:\G\to \C$ and  $S:\G\to M(\G)$, called \emph{counit} and \emph{antipode} respectively, satisfying:
\begin{align*}
(\epsilon \otimes \iota)(\Delta(x)(1\otimes y))&= xy\\
(\iota \otimes \epsilon)((x\otimes 1)\Delta(y))&= xy\\
m(S\otimes \iota)(\Delta(x)(1\otimes y))&= \epsilon(x)y\\
m(\iota\otimes S)((x\otimes 1)\Delta(y))&= \epsilon(y)x,
\end{align*}
where $m:\G\otimes \G\to \G$ is the multiplication map extended to $M(\G)\otimes \G$ and $\G\otimes M(\G)$. The antipode $S$ is an antihomomorphism. 

\smallskip

Any Hopf algebra is a multiplier Hopf algebra (\cite{T}*{Thm.1.3.18}). Moreover, a unital multiplier Hopf algebra $\G$ is a Hopf algebra, as in this case we have $M(\G)=\G$. So we can think of a multiplier Hopf algebra as the natural extension of a Hopf algebra where the underlying algebra is non unital.

Let $\Delta^{op}$ be the opposite comultiplication, i.e. the composite of $\Delta$ with the flip on $\G\otimes \G$. A multiplier Hopf algebra $(\G,\Delta)$ is called \emph{regular} if $(\G,\Delta^{op})$ is also a multiplier Hopf algebra. In this case we also have that $\Delta(\G)(\G\otimes 1)$ and $(1\otimes\G)\Delta(\G)$ are contained in $\G\otimes\G$. A multiplier Hopf algebra $\G$ is regular if and only if the antipode implements a bijection $S:\G\to \G$ (\cite{VD1}*{Prop. 2.7 and 2.8}). From now on, all multiplier Hopf algebras are assumed to be regular, and all the algebras considered have non degenerate products.

Let $(\G,\Delta)$ be a multiplier Hopf algebra. A non zero linear functional $\varphi:\G\to \C$ is called \emph{left invariant} if 
$$(\iota\otimes\varphi)(\Delta(x)(1\otimes y))=\varphi(x)y \quad \forall x,y \in \G.$$
A non zero linear functional $\psi:\G\to \C$ is called \emph{right invariant} if 
$$(\psi\otimes\iota)(\Delta(x)(1\otimes y))=\psi(x)y \quad \forall x,y \in \G.$$
Note that, if $\varphi$ is a left invariant functional, $\varphi\circ S$ is a right invariant functional. In case $\varphi= \varphi \circ S$, we call the multiplier Hopf algebra \emph{unimodular}. It can be shown that invariant functionals are unique (up to scalar), and they are faithful (if $\varphi(x\G)=0$ or $\varphi(\G x)=0$, then $x=0$), see \cite{VD}. 
Following \cite{DVDZ}, we say that an \emph{algebraic quantum group} is a  multiplier Hopf algebra with invariants. We call an algebraic quantum group \emph{finite} if it is finite dimensional.

    Let $(\G,\Delta)$ be an algebraic quantum group. A non zero element $t\in \G$ is called a \emph{left integral} if $xt=\epsilon(x)t$, for all $x\in \G$. Similarly a non zero element $k$ in $\G$ is called \emph{right integral} if $kx=\epsilon(x)k$, for all $x\in \G$. 
If $t\in \G$ is a left integral, it is unique up to scalars and $S(t)$ is a right integral. We say that $(\G,\Delta)$ is \emph{compact} if $\G$ is unital, and we say that $(\G,\Delta)$ is \emph{discrete} if it has an integral.
A compact algebraic quantum group is a Hopf algebra with invariants. If $\G$ is a finite algebraic quantum group, then it is unital. So, finite algebraic quantum groups are just finite dimensional Hopf algebras. In this case, semisimplicity is equivalent to cosemisimplicity, and by \cite{montg}*{2.4.6}, this is the same as being unimodular.

 \begin{rem} The usual Sweedler notation for Hopf algebras can be extended to multiplier bialgebras, see for example \cite{DV, DVDZ, VD1,VDZ1}. From now on we will use this notation.
\end{rem}
 
\subsection{Duality.} Let $(\G,\Delta)$ be an algebraic quantum group, and let $\varphi$ be a left invariant functional of $\G$. For every $x,y\in \G$, consider the functionals $\varphi_x$ and $\varphi^y$ given by $\varphi_x(z)=\varphi(zx)$, $\varphi^y(z)=\varphi(yz) \ \forall z\in \G$, and set $$\hat{\G}=\{\varphi_x \mid x\in \G\}.$$
 By \cite{VD1}*{Section 3}, we can also write $\hat{\G}=\{\varphi^y \mid y\in \G\}$. If we regard $\hat{\G}$ as a subspace of $\G^*$, the dual of the product and the coproduct of $\G$ gives a product and a coproduct $\hat{\Delta}$ on $\hat{\G}$ such that $(\hat{\G},\hat{\Delta})$ becomes an algebraic quantum group. 
\begin{teo}(\cite{VD}*{Thm. 4.12}) \label{dualqg}  If $(\G,\Delta)$ is an algebraic quantum group, then $(\hat{\G},\hat{\Delta})$ is again an algebraic quantum group, and there is a canonical isomorphism  $$(\hat{\hat{\G}}, \hat{\hat{\Delta}})\cong (\G,\Delta).$$
\end{teo}

If $\G$ is a finite algebraic quantum group, $\hat{\G}$ coincides with $\G^*$. So, the previous Theorem generalizes the duality of finite dimensional Hopf algebras.
\subsection{Actions and coactions}\label{actions}

Let $(\G,\Delta)$ be an algebraic quantum group, and $A$ a $\C$-algebra, with or without unit, but with non-degenerate product.

We say that $A$ is a \emph{left $\G$-module algebra} if it satisfies the following conditions:
\begin{enumerate}
\item \label{unital}$A$ is a unital left $\G$-module,  here $A$ unital means $\G A=A$; where \\
$\G A=\{\sum x_i\rightharpoonup a_i \mid x_i\in\G, a_i\in A\}$,
    \item $x\rightharpoonup aa'= \sum (x_1\rightharpoonup a)(x_2\rightharpoonup a')$, $\forall x\in \G, a, a'\in A$.
\end{enumerate}
\emph{Right} $\G$-\emph{module algebras} are defined in a similar way.

\smallskip

Given a left $\G$-module algebra $A$ we can define the \emph{smash product} algebra $A\# \G$ as follows: the underlying vector space of $A\# \G$ is $A\otimes \G$, and the product is given by:
$$(a\#x)(a'\# y)= \sum a(x_1\rightharpoonup a')\# x_2 y.
$$

\smallskip
We say that $A$ is a \emph{right $\G$-comodule algebra} if there exists an injective homomorphism $\delta: A\to M(A\otimes \G)$ such that
\begin{itemize}
\item $\delta(A)(1\otimes \G)\subseteq A\otimes \G$ and $(1\otimes \G)\delta(A)\subseteq A\otimes \G$,
\item $(\delta \otimes \iota_{\G})\circ \delta=(\iota_A\otimes\Delta)\circ\delta$, i.e. the following diagram commutes:
$$
\xymatrix{A\ar[r]^\delta\ar[d]_\delta & M(A\otimes \G)\ar[d]^{\delta\otimes\iota_{\G}}\\M(A\otimes \G)\ar[r]_{\iota_A\otimes \Delta}&M(A\otimes \G\otimes \G)}
$$
\end{itemize}
In this case we will call $\delta$ a \emph{right coaction} of $\G$ on $A$. Similarly one can define \emph{left} $\G$-\emph{comodule algebras}.

We will also use the Sweedler notation for comodule algebras: for all $a\in A$ we will write \\$\delta(a)=\sum a_0\otimes a_1$ and $(\delta \otimes\iota_{\G})(\delta(a))=\sum a_0\otimes a_1\otimes a_2$.

\smallskip

 Let $(\G,\Delta)$ be an algebraic quantum group, and fix $\varphi$ a left invariant functional of $\G$. Suppose we have a right coaction $\delta:A\to M(A\otimes \G)$ of $\G$ on an algebra $A$. Then, we can define a left action of $\hat{\G}$ on $A$ as follows. First recall that every element $\alpha \in \hat{\G}$ can be uniquely expressed as $\alpha=\varphi^y$, for some $y\in \G$. The action is given by:
$$
\alpha\rightharpoonup a = (\iota_A\otimes\varphi)((1\otimes y)\delta(a)), \hspace{0.5cm} \forall\hspace{0.1cm} \alpha=\varphi^y\in \hat{\G}, a\in A.
$$
Using Sweedler notation, we can express the action as $\alpha\rightharpoonup a=\sum \alpha(a_1)a_0$.
Moreover the converse is also true: if $A$ is a left $\hat{\G}$-module algebra, then it is a right $\G$-comodule algebra, and we have the following result, which again generalizes the case of finite dimensional Hopf algebras:
\begin{teo}\label{dualaction}(\cite{VDZ1}*{Thm. 3.3})
    Let $(\G,\Delta)$ be an algebraic quantum group and $A$ an algebra. Then $A$ is a (right/left) $\G$-module algebra iff $A$ is a (left/right) $\hat{\G}$-comodule algebra.
\end{teo}

\begin{exa}\label{dualcoaction} Let $(\G,\Delta)$ be an algebraic quantum group and $\varphi$ a left invariant functional. The map $\Delta:\G\to M(\G\otimes \G)$ gives $\G$ the structure of a right $\G$-comodule algebra. The dual left action of $\hat{\G}$ on $\G$ in this case is:
$$
\alpha\rightharpoonup x = (\iota_{\G}\otimes\varphi)((1\otimes y)\Delta(x)), \hspace{0.5cm} \forall\hspace{0.1cm} \alpha=\varphi^y\in \hat{\G}, x\in \G.
$$
    \end{exa}

 \subsection{Compact subgroups and induced actions} 
The notion of quantum subgroup in this setting is not quite clear (see \cite{vaes, dkss} for the definition of \emph{closed quantum subgroup} in the context of locally compact quantum groups and \cite{Chris} for the definition of quantum subgroup for Hopf algebras). Here we consider the following definition of \emph{algebraic compact quantum subgroup}.

\begin{defi}\label{cqs} Let $\G$ and $\mathcal{U}$ be algebraic quantum groups. We call $\mathcal{U}$ a \emph{compact quantum subgroup} of $\G$ if $\mathcal{U}$ is compact, unimodular and there exists a surjective morphism $\pi:\G\to \mathcal{U}$.
\end{defi}
Let $\G$ be an algebraic quantum group and $\U$ a compact quantum subgroup. Note that 
\begin{equation}\label{coaction}\lambda_{\U}=(\pi \otimes\iota_{\G})\Delta: \G\to M(\U\otimes\G)
\end{equation}
defines a structure of left $\U$-comodule algebra on $\G$. Then, by Theorem \eqref{dualaction}, $\G$ has a structure of right $\hat{\U}$-module algebra with the action given by:
\begin{equation}\label{dualbeta}
    x\leftharpoonup \beta= \sum \beta(\pi(x_1))x_2.
\end{equation} 

Suppose that $A$ is a left $\hat{\U}$-module algebra with unit, we want to construct an algebra with an induced left $\hat{\G}$-action. We can define a right action of $\hat{\U}$ on $A\otimes\G$ as follows: 
\begin{equation}\label{actionontensor}
a\otimes x \leftharpoonup \beta= \sum\bar{S}\beta_1 \rightharpoonup a \otimes x\leftharpoonup \beta_2 ,
\end{equation}
here $\overline{S}$ is the antipode of $(\U,\Delta^{op})$. This action endows $A\otimes\G$ with a structure of a right $\hat{\U}$-module algebra. Following \cite{DVDZ}*{Prop. 4.11}, the algebra of invariants for this action is the subalgebra of $M(A\otimes\G)$ defined by:

$$(A\otimes\G)^{\hat{\U}}=\left\{\begin{array}{lll}m\in M(A\otimes\G) \mid & (my)\leftharpoonup \beta =m(y\leftharpoonup \beta) &  \forall y\in A\otimes\G, \forall \beta\in \hat{\U}\\ &
(ym)\leftharpoonup \beta=( y\leftharpoonup \beta)m,&
\end{array}\right\}.
$$
\bigskip

 By Example \eqref{dualcoaction}, $A\otimes\G$ has a structure of left $\hat{\G}$-module algebra with the action: 
\begin{equation}\label{leftaction}
    \alpha\rightharpoonup (a\otimes x)=a\otimes (\alpha\rightharpoonup x).
\end{equation}
Moreover, as the actions \eqref{actionontensor} and \eqref{leftaction} commute, we can define a left action of $\hat{\G}$ on $(A\otimes\G)^{\hat{\U}}$ as follows (see \cite{DVDZ}*{Proposition 4.7}):
\begin{align*}
    (\alpha\rightharpoonup m)x=& \sum\alpha_1\rightharpoonup \left(m(S(\alpha_2)\rightharpoonup x) \right)\\
    x(\alpha\rightharpoonup m)=& \sum \alpha_2 \rightharpoonup\left( (S^{-1}(\alpha_1)\rightharpoonup x)m\right)
\end{align*}

for every $\alpha \in \hat{\G}$, $m\in (A\otimes\G)^{\hat{\U}}$, $x\in A\otimes \G$. 

\bigskip

In the following sections we will study two families of algebraic quantum groups:  $\G=C_c(G)$, for $G$ a group,  and $\G=\H$ a finite algebraic quantum group (i.e. a finite dimensional Hopf algebra). In each case we consider a subalgebra of $(A\otimes\G)^{\hat{\U}}$ that has an induced left $\hat{\G}$-action.

\section{Induction of actions for \texorpdfstring{$C_c(G)$}{Cc(G)}} \label{sec2}

Let $G$ be a group. Write $e$ for the neutral element and $\chi_g$ for the characteristic function of the element $g\in G$. Recall that
\begin{equation}\label{ccg}
C_c(G)=\left\{\sum_{g\in G}{\lambda_g\chi_g}\mid \lambda_g\in\C,\lambda_g \neq 0 \text{ for a finite amount of } g\right\}
\end{equation} is an algebraic quantum group with the following structure:
\begin{multicols}{2}
\begin{align*}
\chi_g \chi_h&=\delta_{g,h}\chi_g\\
\epsilon(\chi_g)&=\delta_{g,e}\\
S(\chi_g)&=\chi_{g^{-1}}
\end{align*}

\columnbreak
\begin{align*}
\Delta_G(\chi_g)&=\sum_{l\in G}\chi_{l}\otimes \chi_{l^{-1}g}\\
\varphi(\chi_g)&=\psi(\chi_g)=1
\end{align*}
\end{multicols}

Note that $\chi_e\in C_c(G)$ is a left (and right) integral, so $C_c(G)$ is discrete.

The dual quantum group $\widehat{C_c(G)}$ can be identified with the Hopf algebra
$$\C G=\left\{\sum_{g\in G} {\lambda_gg \mid \lambda_g}\in\C,\lambda_g \neq 0 \text{ for a finite amount of } g\right\}$$
 with its usual counit and comultiplication
$$\epsilon(g)=1 \qquad  \Delta(g)=g\otimes g.$$
The invariant functionals are given by:
$$\phi(g)=\psi(g)=\delta_{g,e}. $$

Let $H\subseteq G$ be a finite subgroup. In this case, both $C_c(H)$ and $\C H$ are finite dimensional Hopf algebras. Note that $$C_{c}(H)=C(H)=\left\{\sum_{h\in H}{\lambda_h\chi_h}\mid \lambda_h\in\C \right\}.$$
If we write $\pi:C_c(G)\to C(H)$ for the projection, we have that $C(H)$ is a compact quantum subgroup of $C_c(G)$. The left coaction \eqref{coaction} of $C(H)$ on $C_c(G)$  induces the following right action of $\C H$ on $C_c(G)$:

$$ \sum_{g\in G}\lambda_g\chi_g \leftharpoonup h = \sum_{g\in G}\lambda_g\chi_{h^{-1}g}.$$
We can also endow $C_c(G)$ with the structure of a left $\C G$-module algebra with the action:
$$t\rightharpoonup \sum_{g\in G}\lambda_g\chi_g  = \sum_{g\in G}\lambda_g\chi_{gt^{-1}}.  $$

Let $A$ be a left $\C H$-module algebra with unit. The tensor product $A\otimes C_c(G)=C_c(G,A)=\{f:G\to A \mid f(g)\neq 0 \text{ for a finite amount of } g \}$ has a structure of right $\C H$-module algebra with the action (see \eqref{actionontensor}):
$$ (f\leftharpoonup h)(g)=h^{-1}\rightharpoonup f(hg),$$
for every $f\in C_c(G,A)$, $h \in H, g\in G$.

In this case, the action can be extended to the multiplier algebra $M(A\otimes C_c(G))=C(G,A)$. The subalgebra of invariants $C(G,A)^{\C H}$ can be endowed with the following left action of $\C G$:
$$(t\rightharpoonup f)(g)=f(gt), $$
for every $f\in C(G,A)$, $t,g\in G$.

\begin{rem} Let $p: G\to G/H$ be the projection and consider the $G$-algebras defined in \cite{CE}:
$$\operatorname{BigInd}_H^G(A) =\{ f:G\to A \mid h\rightharpoonup f(g)=f(hg)\hspace{0.2cm}  \forall h\in H, \hspace{0.2cm}  \forall g\in G\}$$ and
$$\operatorname{Ind}_H^G(A)=\{f\in \operatorname{BigInd}_H^G(A) \mid |p(\operatorname{supp}(f))|<\infty \},$$
here $\operatorname{supp}(f)=\{g\in G \mid f(g)\neq 0 \}$. Note that, if $H$ is finite, every element of $\operatorname{Ind}_H^G(A)$ has finite support. Hence, for $H$ finite, we have the following identifications:
\begin{align*}
   \operatorname{BigInd}_H^G(A)&=C(G,A)^{\C H}\\
\operatorname{Ind}_H^G(A)&=C_c(G,A)^{\C H}.     
\end{align*}
\end{rem}

\subsection{Morita context} In the following we will use the notion of Morita equivalence between rings with \emph{local units.} Recall that a ring $R$ has local units if for every finite subset $T\subseteq R$ there exists an idempotent $e\in R$ such that $te=et=t$, $\forall t \in T$. Examples of rings with local units are unital rings and discrete algebraic quantum groups (\cite{VDZ2}*{Prop. 3.1}). 

If $R$ is a ring with local units, we write $\operatorname{RMod}$ for the category of unital left $R$-modules (see \eqref{unital} in Section \ref{actions}). We say that two rings with local units $R$ and $S$ are \emph{Morita equivalent} if the categories $\operatorname{RMod}$ and $\operatorname{SMod}$ are equivalent (\cite{a}). 
 A \emph{Morita context} between $R$ and $S$ is given by a unital $R$-$S$-bimodule $P$, a unital $S$-$R$-bimodule $Q$, and two bimodule homomorphisms 
$$\Gamma:P\otimes_{S}Q \to R \qquad \Lambda:Q\otimes_{R}P\to S$$
satisfying:
\begin{align}\label{comp}
    \Gamma(p\otimes q)\rightharpoonup p' &= p\leftharpoonup \Lambda(q\otimes p'), \qquad\forall p,p' \in P, q\in Q.\\
\Lambda(q\otimes p)\rightharpoonup q'&= q \leftharpoonup \Gamma(p\otimes q'), \qquad\forall p \in P, q,q' \in Q.\notag
\end{align}
Two algebras are Morita equivalent if and only if there exists a Morita context between them with both homomorphisms surjective (\cite{willie}*{Prop. 5.1.8}).

\smallskip

 The discrete algebraic quantum group $C_c(G)$ has local units, so the invariant algebra $C_c(G,A)^{\C H}$ also has local units. Hence, the smash product $C_c(G,A)^{\C H}\#\C G$ has local units as well (see the proof of Lemma 5.2 in \cite{VDZ1}). In the next proposition we show that the algebras $C_c(G,A)^{\C H}\# \C G$ and $A \# \C H$ are Morita equivalent by constructing bimodules which implement a Morita context (compare \cite{CE}*{Theorem 10.4.5}).

\begin{prop} \label{moritacoeff}Let $G$ be a group, $H$ a finite subgroup and $A$ a left $\C H$-module algebra with unit. There is a Morita equivalence between $C_c(G,A)^{\C H}\# \C G$  and $A \# \C H$.
\end{prop}

\begin{proof} The bimodules implementing the Morita context are $C_c(G,A)$ and  $A\otimes \C G$. 
 $C_c(G,A)$ is an $(C_c(G,A)^{\C H}\#\C G, A\#\C H)$-bimodule with the following actions:
 \begin{align*}
        (f \leftharpoonup a\# h)(g) &= h^{-1}\rightharpoonup (f(hg)a)\\
       ( F\# t \rightharpoonup  f)(g) &=F(g)f(gt)
    \end{align*}

     $A\otimes \C G$ is an $(A\# \C H, C_c(G,A)^{\C H}\#\C G)$-bimodule with the following actions:
     \begin{align*}
     b\# h\rightharpoonup a\otimes g &= b(h\rightharpoonup a)\otimes hg\\
     a\otimes g \leftharpoonup F\# t &= aF(g)\otimes   gt
 \end{align*}

 The bimodule homomorphisms are defined by:

$$ \Lambda: (A\otimes \C G) \otimes_{C_c(G,A)^{\C H}\#\C G} C_c(G,A) \to A\#\C H
$$

$$\Lambda((a\otimes t) \otimes f)=\sum_{h\in H}a (h\rightharpoonup f(h^{-1}t))\# h $$
$$\Gamma: C_c(G,A)\otimes_{A\# \C H}(A\otimes\C G)\to C_c(G,A)^{\C H}\#\C G $$

$$ \Gamma (f\otimes(a\otimes t))=\sum_{s \in \operatorname{supp}(f)} F_{s,a}\# s^{-1}t$$

where $F_{s,a}(x)= \sum_{h} (h\rightharpoonup f(h^{-1}x)a)\delta_{x,hs}.$

We claim that both homomorphisms satisfy conditions \eqref{comp} and are surjective. In fact, given $a\# h \in A\#\C H$ and 

$f\# k \in C_c(G,A)^{\C H}\#\C G$ we have:

\begin{align*}
 a\# h = &\Lambda \left( (a\otimes h) \otimes (1_A\otimes \chi_e)\right),\\
f\# k =& \Gamma \left(\sum_{s\in \operatorname{supp}(f)} f\chi_s \otimes (1_A\otimes sk) \right).
\end{align*}

\end{proof}

\section{Induction of actions for finite algebraic quantum groups}\label{sec3}

Let $(\H,\Delta_{\H})$ be a finite algebraic quantum group, i.e. a finite dimensional Hopf algebra. In this case, the dual $\hat{\H}$ is the space of linear functionals $\H^*=\operatorname{Hom}_{\C}(\H,\C)$ with the following multiplication and comultiplication:
\begin{eqnarray*}
    (f\star g)(h)&=& (f\otimes g)(\Delta_{\H}(h))\\
\Delta_{\hat{\H}}(f)(h\otimes k)&=&f(hk)
\end{eqnarray*}
$\forall f,g \in \hat{\H}$, $\forall h,k \in \H$.
If $t\in \H$ is a left integral, $\hat{\varphi}=\text{ev}_t$ is a left invariant functional of $\H$. We will consider $\H$ as an $\hat{\H}$-bimodule algebra with the actions:

\begin{align*}
 \alpha\rightharpoonup h&=\sum h_{1}\alpha(h_2)\cr
 h\leftharpoonup\mu &= \sum \mu(h_1)h_2
\end{align*}
$\forall h\in \H$, $\forall \alpha, \mu \in \hat{\H}$.

Suppose we have a surjective homomorphism $\pi:\H \to \U$, where $\U$ is a finite algebraic quantum group. Then we can endow $\H$ with the structure of a right $\hat{\U}$-module algebra with the following action (see \eqref{dualbeta}) $$h\leftharpoonup\beta= \sum \beta(\pi(h_1))h_2.$$

 Given a left $\hat{\U}$-module algebra $A$ with unit, we can endow  $A\otimes\H$ with a structure of a right $\hat{\U}$-module (see \eqref{actionontensor}):
\begin{equation}\label{action}
a\otimes h\leftharpoonup\beta=\sum \overline{S}(\beta_1)\rightharpoonup a\otimes h\leftharpoonup \beta_2.
 \end{equation}

    The algebra of invariants $(A\otimes\H)^{\hat{\U}}$ has a structure of left $\hat{\H}$-module algebra with action $$\alpha\rightharpoonup\sum a_i\otimes h_i= \sum a_i\otimes \alpha\rightharpoonup h_i.$$ 
    %If $\U$ is a compact quantum subgroup of $\H$, we write $\tilde{A}$ for $(A\otimes\H)^{\hat{\U}}$.

\begin{rem} For Hopf algebras there is a definition of \emph{quantum subgroup} (see \cite{Chris}*{Def. 3.3}), and it coincides with Definition \ref{cqs} in the context of finite algebraic quantum groups.
    \end{rem}
    In the rest of this section we will show that for $\U$ a compact quantum subgroup of $\H$, we have a Morita equivalence between $A\#\hat{\U}$ and $(A\otimes\H)^{\hat{\U}}\#\hat{\G}$.

\begin{lema}\label{context}
    Let $\H$ and $\U$ be finite algebraic quantum groups, $\pi:\H\to \U$ a surjective homomorphism and $A$ a left $\hat{\U}$-module algebra. Then, $A\otimes \H$ is a right $A\#\hat{\U}$-module and there is a Morita context between $A\#\hat{\U}$ and $\operatorname{End}_{A\#\hat{\U}}(A\otimes\H)$. 
\end{lema}

\begin{proof}
   We can endow $A\otimes \H$ with structure of  right $A\#\hat{\U}$-module by defining:
    \begin{equation}\label{raction}a\otimes h \leftharpoonup b\#\beta:= \sum\overline{S}(\beta_1)\rightharpoonup (ab)\otimes h \leftharpoonup \beta_2.\end{equation}
Hence, $A\otimes\H$ becomes an $\left(\operatorname{End}_{A\#\hat{\U}}(A\otimes \H), A\#\hat{\U}\right)$-bimodule, with the left action of $\operatorname{End}_{A\#\hat{\U}}(A\otimes \H)$ given by evaluation.

Let us now consider $\operatorname{Hom}_{A\#\hat{\U}}(A\otimes\H, A\#\hat{\U})$, which is an $\left(A\#\hat{\U},\operatorname{End}_{A\#\hat{\U}}(A\otimes \H)\right)$-bimodule with the following actions:
  \begin{align*}
       (b\#\beta\rightharpoonup f)(a\otimes h)=&(b\#\beta)f(a\otimes h)\cr 
       (f\leftharpoonup T)(a\otimes h)=& f(T(a\otimes h)).
   \end{align*}
The bimodule homomorphisms implementing the Morita context are:
 \begin{align*}
       \Lambda: \operatorname{Hom}_{A\#\hat{\U}}(A\otimes\H, A\#\hat{\U}) &\otimes_{\operatorname{End}_{A\#\hat{\U}}(A\otimes \H)}A\otimes\H\to A\#\hat{\U}\cr
    \Lambda(f\otimes (a\otimes h))&=f(a\otimes h) 
      \end{align*}
   \begin{align*}   
      \Gamma: A\otimes \H  \otimes_{A\#\hat{\U}}  \operatorname{Hom}_{A\#\hat{\U}}&(A\otimes\H, A\#\hat{\U})\to \operatorname{End}_{A\#\hat{\U}}(A\otimes \H)\cr
      \Gamma((a\otimes h)\otimes f) (a'\otimes h')&=(a\otimes h)\leftharpoonup f(a'\otimes h').
   \end{align*}

It is straightforward to check that the compatibility conditions \eqref{comp} are satisfied.

\end{proof}

\begin{coro}\label{trivialaction}
    Let $\H$ be finite algebraic quantum group an $\U$ a compact quantum subgroup. For any unital algebra $A$, with trivial action of \hspace{0.01cm} $\hat{\U}$, there is a Morita equivalence between $A\#\hat{\U}$ and $\operatorname{End}_{A\#\hat{\U}}(A\otimes\H)$. 
    \end{coro}

\begin{proof} 
  In this case, the smash product $A\#\hat{\U}$ coincides with the tensor product $A\otimes \hat{\U}$. By the previous lemma, we have a Morita context between $A\otimes\hat{\U}$ and $\operatorname{End}_{A\otimes\hat{\U}}(A\otimes\H)$. In order to obtain a Morita equivalence it is enough to show that $A\otimes \H$ is a finitely generated projective generator as a right $A\otimes\hat{\U}$-module (\cite{willie}*{Thm.5.2.7}).

Let $\pi:\H\to\U$ be the surjective homomorphism of the Definition \ref{cqs} and let $\gamma$ be the (left and right) invariant functional of $\U$. We have a surjective homomorphism of $A\otimes\hat{\U}$-modules $p:A\otimes\H \to A\otimes\hat{\U}$, given by $p(a\otimes h)=a\otimes\gamma^{S(\pi(h))}$. This proves that $A\otimes\H$ is a generator.

Now consider a diagram of right $A\otimes\hat{\U}$-module homomorphisms,  with $g$ surjective:

\begin{equation}\label{diag}
    \xymatrix{& N\ar@{->>}[d]^{g}\\A\otimes\H \ar[r]_f&M}\end{equation}
By restricting the action of $A\otimes\hat{\U}$ to $1_A\otimes\hat{\U}$, \eqref{diag} becomes a diagram of $\hat{\U}$-module homomorphisms. Write $i:1_A\otimes \H \to A\otimes\H$ for the inclusion. Then, as the unimodularity of $\U$ implies the semisimplicity of $\hat{\U}$, there is a homomorphism of $\hat{\U}$-modules $\tilde{f}:1_A\otimes\H\to N$ such that $g\tilde{f}=fi$. Define $\bar{f}:A\otimes\H \to N$ by $\bar{f}(a\otimes h)= \tilde{f}(1\otimes h)\leftharpoonup a\otimes\epsilon$. It is easy to check that $\bar{f}$ is an $A\otimes\hat{\U}$-module homomorphism which satisfies $g\bar{f}=f$. Thus $A\otimes\H$ is projective as $A\otimes\hat{\U}$-module.
Finally, as $\H$ has finite dimension, $A\otimes\H$ is finitely generated as $A\otimes\hat{\U}$-module.
\end{proof}

\begin{teo}\label{morita}
  Let $\H$ be finite algebraic quantum group, $\U$ a compact quantum subgroup and $A$ a left $\hat{\U}$-module algebra with unit. The algebras $A\#\hat{\U}$ and $(A\otimes\H)^{\hat{\U}}\#\hat{\H}$ are Morita equivalent, where  $(A\otimes\H)^{\hat{\U}}$ is the invariant subalgebra for the action \eqref{action}. 
 
 \end{teo}
  
%  , the algebra of invariants $\tilde{A}=(A\otimes\H)^{\hat{\U}}$ for the action \ref{action} has a structure of left $\hat{\H}$-module algebra and 

 %Given a unital left $\hat{\U}$-module algebra $A$, 

\begin{proof} Let us first note that $(A\otimes\H)^{\hat{\U}}\#\hat{\H}$ can be identified with a subalgebra of $\operatorname{End}_{A\#\hat{\U}}(A\otimes\H)$:
	
	$$\left(\sum_i (a_i\otimes h_i)\#\alpha \right)(b\otimes k)= \sum_i a_ib \otimes h_i(\alpha\rightharpoonup k),$$
	for every $\sum_i (a_i\otimes h_i)\#\alpha \in (A\otimes\H)^{\hat{\U}}\#\hat{\H}$, and  $b\otimes k \in A\otimes\H$.\\
	
	On the other hand, using the isomorphisms:
	\begin{itemize}
		\item[-] $\displaystyle{\operatorname{End}_{A\#\hat{\U}}(A\otimes\H)\cong\operatorname{Hom}_{\hat{\U}}(\H,A\otimes\H)},$
		\item[-] $\displaystyle{\operatorname{Hom}(\H,A\otimes\H)\cong A\otimes\operatorname{End}(\H)}$  (see \cite{nb}*{Prop.10, Ch.I, Secc. 2.9}),
		\item[-] $ A\otimes\operatorname{End}(\H)\cong A\otimes (\H\otimes\hat{\H})$
	\end{itemize}
	 %$$\operatorname{End}_{A\#\hat{\U}}(A\otimes\H)\cong\operatorname{Hom}_{\hat{\U}}(\H,A\otimes\H), \quad \operatorname{Hom}(\H,A\otimes\H)\overset{(1)}{\cong} A\otimes\operatorname{End}(\H)\cong A\otimes (\H\otimes\hat{\H})$$ (see \cite{nb}*{Prop.10, Ch.I, Secc. 2.9} for (1)), 
	 every element $T \in \operatorname{End}_{A\#\hat{\U}}(A\otimes\H)$ can be written as $T=\sum_{i,j} a_i\otimes h_{ij}\otimes \alpha_{ij}$ for some $a_i\in A$, $h_{ij}\in\H$ and $\alpha_{ij}\in \hat{\H}$. With this notation we have:
	$$T(k)=\sum_{i,j}a_i\otimes h_{ij}\alpha_{ij}(k),$$
	$\forall k\in \H$.
	Moreover, as $T$ is a homomorphism of $\hat{\U}$-modules, it verifies:
	\begin{align}\label{beta}
		T(k\leftharpoonup \beta)= &T(k)\leftharpoonup \beta\notag\\
		\sum_{i,j}a_i\otimes h_{ij}\alpha_{ij}(k\leftharpoonup\beta)=& \sum_{ij}\bar{S}(\beta_1)\rightharpoonup a_i \otimes (h_{ij}\alpha_{ij}(k))\leftharpoonup\beta_2 .
	\end{align}
	Note that for each $i$, $f_i=\sum_j h_{ij}\otimes \alpha_{ij} \in \operatorname{End}(\H)$ using the third isomorphism previously described, and by \cite{M}*{Thm. 1.7.4}, if $\left\{ v^l\right\}_{l=1}^n$ is a basis of $\H$ and $\left\{\delta^l\right\}_{l=1}^n$ is the dual basis of $\hat{\H}$ we have:
	$$f_i(k)=\sum_{j,l}h_{ij}\bar{S}(v^l_1)\alpha_{ij}(v^l_2)(\delta^l\rightharpoonup k)= \sum_l \underbrace{\left( \sum_j h_{ij}\bar{S}(v^l_1)\alpha_{ij}(v^l_2)\right)}_{q_i^l}(\delta^l\rightharpoonup k).$$
	Hence, to show that $T$ belongs to $(A\otimes\H)^{\hat{\U}}\#\hat{\H}$ it suffices to prove that for every $l$, $\sum_i a_i \otimes q_i^l \in (A\otimes\H)^{\hat{\U}}$, that is
	$\bigl(\sum_i a_i \otimes q_i^l\bigr)\leftharpoonup\beta=\epsilon(\beta)\sum_i a_i \otimes q_i^l$.\\
	Based on the definitions of $q_i^l$ and $\leftharpoonup$, it follows that 
	$$\Bigl(\sum_i a_i \otimes q_i^l\Bigr)\leftharpoonup\beta=\sum_i \bar{S}\beta_1\rightharpoonup a_i\otimes \Bigl[\Bigl(\sum_jh_{ij}\bar{S}\bigl(v_1^l\bigr)\alpha_{ij}\bigl(v_2^l\bigr)\Bigr)\leftharpoonup\beta_2\Bigr]$$
	Using the structure of right $\hat{\U}$-module algebra of $\H$ it coincides with
	$$\Bigl[\sum_{i,j}\bigl(\bar{S}\beta_1\rightharpoonup a_i\bigr)\otimes \bigl(h_{ij}\alpha_{ij}\bigl(v_2^l\bigr)\leftharpoonup\beta_2\bigr)\Bigr]\Bigl(1\otimes \bar{S}\bigl(v_1^l\bigr)\leftharpoonup\beta_3\Bigr)$$ 
	Applying (\ref{beta}) the preceding expression can be reformulated as
	$$\Bigl[\sum_{i,j} a_i\otimes h_{ij}\alpha_{ij}\bigl(v_2^l\leftharpoonup\beta_1\bigr)\Bigr]\Bigl(1\otimes \bar{S}\bigl(v_1^l\bigr)\leftharpoonup\beta_2\Bigr)$$ 
	After applying the definitions of the actions and utilizing the properties of counit and antipode, we arrive at the expression
	$$\epsilon(\beta)\sum_ia_i\otimes \Bigl[\sum_jh_{ij}\alpha_{ij}\bigl(v_2^l\bigr)\bar{S}\bigl(v_1^l\bigr)\Bigr]=\epsilon(\beta)\sum_ia_i\otimes q_i^l.$$	
	Then we can identify $(A\otimes\H)^{\hat{\U}}\#\hat{\H}$ with $\operatorname{End}_{A\#\hat{\U}}(A\otimes\H)$. This identification, together with Lemma \ref{context}, shows that there is a Morita context between $A\#\hat{\U}$ and $(A\otimes\H)^{\hat{\U}}\#\hat{\H}$. 
	Upon replacing $A$ with the subalgebra $A^{\hat{\U}}$, the Morita context transforms into a Morita equivalence (Corollary \ref{trivialaction}).
	  In this case the bimodules of the context are $M_{A^{\hat\U}}=A^{\hat\U}\otimes\H$ and $N_{A^{\hat\U}}=\operatorname{Hom}_{A^{\hat{\U}}\otimes\hat{\U}}(A^{\hat{\U}}\otimes\H,A^{\hat\U}\otimes \hat{\U})=\operatorname{Hom}_{\hat{\U}}(\H,A^{\hat\U}\otimes \hat{\U})$, and the surjective bimodule homomorphisms:  
	$$\Lambda_{A^{\hat{\U}}}:N_{A^{\hat{\U}}}\otimes_{\left(A^{\hat{\U}}\otimes\H\right)^{\hat{\U}}\#\hat{\H} }M_{A^{\hat{\U}}}\to A^{\hat{\U}}\otimes\hat{\U},$$	
	$$\Gamma_{A^{\hat{\U}}}: M_{A^{\hat{\U}}} \otimes_{A^{\hat{\U}}\otimes\hat{\U}} N_{A^{\hat{\U}}}\to \left(A^{\hat{\U}}\otimes\H\right)^{\hat{\U}}\#\hat{\H}.$$	
	Let us now consider the Morita context for $A$, and write:
	$M_A= A\otimes \H$ and $N_A=\operatorname{Hom}_{A\#\hat{\U}}(A\otimes\H,A\#\hat{\U})=\operatorname{Hom}_{\hat{\U}}(\H,A\#\hat{\U})$ for the bimodules, and $$\Lambda_A:N_A\otimes_{(A\otimes\H)^{\hat{\U}}\#\hat{\H} }M_A\to A\#\hat{\U}$$
	$$\Gamma_A: M_A \otimes_{A\#\hat{\U}} N_A\to (A\otimes\H)^{\hat{\U}}\#\hat{\H}$$
	for the bimodule homomorphisms.
	
	Note that $M_{A^{\hat{\U}}}\subseteq M_A$ and $N_{A^{\hat{\U}}}\subseteq N_A$. Moreover, for every $a\#\beta \in A\#\hat{\U}$ and every $\sum_i(b_i\otimes h_i)\#\alpha \in (A\otimes\H)^{\hat{\U}}\#\hat{\H}$ we have:
	\begin{align*}
		a\#\beta= &(a\#\epsilon)(1\#\beta), \quad \text{with } 1\#\beta \in A^{\hat{\U}}\otimes\hat{\U},\\\sum_i(b_i\otimes h_i)\#\alpha= &\sum_i\left((b_i\otimes h_i)\#\epsilon\right)\left((1\otimes 1)\#\alpha\right),  \text{with } (1\otimes1)\#\alpha \in (A\otimes\H)^{\hat{\U}}\#\hat{\H}.
	\end{align*}
	
	Using the fact that both $\Lambda_A$ and $\Gamma_A$ are bimodule homomorphisms, and that $\Lambda_{A^{\hat{\U}}}$ and $\Gamma_{A^{\hat{\U}}}$ are surjective, we can conclude that $\Lambda_A$ and $\Gamma_A$ are also surjective. This proves that $(A\otimes\H)^{\hat{\U}}\#\hat{\H}$ and $A\#\hat{\U}$ are Morita equivalent.

\end{proof}

\end{document}